\documentclass[11pt, a4paper,reqno]{amsart}
\usepackage[T1]{fontenc}
\usepackage{amsaddr}
\usepackage{yhmath,mathrsfs,amsthm,amsmath,amssymb,amsfonts, enumerate,lipsum, appendix, mathtools, bm}
\usepackage{bbold}
\usepackage[mathscr]{euscript}
\usepackage[normalem]{ulem} 
\usepackage{graphicx}
\usepackage{pgf,tikz}
\usepackage{tkz-euclide}
\usepackage{graphicx}
\usepackage{subcaption}
\usetikzlibrary{shapes.geometric}
\usetikzlibrary{arrows,hobby}
\usepackage{enumitem}
\usepackage[english]{babel}

\allowdisplaybreaks

\usepackage[sort,numbers]{natbib}
\usepackage{hyperref}
\hypersetup{
	colorlinks=true,
	linkcolor=blue,
	filecolor=magenta,      
	urlcolor=red,
	citecolor=red,
}
\usepackage[margin=0.90in]{geometry}
\allowdisplaybreaks
\usepackage{orcidlink}
\newtheorem{theorem}{Theorem}[section]
\newtheorem{lemma}[theorem]{Lemma}
\newtheorem{proposition}[theorem]{Proposition}
\newtheorem{definition}[theorem]{Definition}
\theoremstyle{definition}

\newtheorem{remark}[theorem]{Remark}
\numberwithin{equation}{section}
\usepackage[hyperpageref]{backref}

\newcommand*\N{\mathcal{N}}
\newcommand{\al} {\alpha}

\newcommand{\Om} {\Omega}

\newcommand{\Gr} {\nabla}
\newcommand{\no} {\nonumber}
\newcommand{\noi} {\noindent}

\newcommand{\ra} {\rightarrow}

\def\dx{\,{\rm d}x}
\def\dy{\,{\rm d}y}

\def\C{{\mathcal C}}

\def\R{{\mathbb R}}
\def\N{{\mathbb N}}

\def\({{\Big(}}
\def\){{\Big)}}

\def\dx{\,{\rm d}x}

\DeclarePairedDelimiter\abs{\lvert}{\rvert}%
\DeclarePairedDelimiter\norm{\lVert}{\rVert}%
\def\wps{{W_0^{s,p}(\Om)}}

\newcommand{\ii}{\int\limits_}
\title[Faber-Krahn inequality for mixed local-nonlocal operator]{Strict Faber-Krahn type inequality for the mixed local-nonlocal operator under polarization$^\bigstar $}
\author[K. Ashok~Kumar and N. Biswas]{K Ashok Kumar\,\orcidlink{0000-0002-4098-3705} \and Nirjan Biswas$^\dagger$\,\orcidlink{0000-0002-3528-8388}}
\address{\rm Tata Institute of Fundamental Research Centre For Applicable Mathematics \\ Post Bag No 6503, Sharada Nagar, Bengaluru 560065, India}
\email[Ashok]{ashok23@tifrbng.res.in, srasoku@gmail.com} \email[Nirjan]{nirjan22@tifrbng.res.in}
\thanks{$^\bigstar $This work is partially supported by the Department of Atomic Energy, Government of India, under project no. 12-R$\&$D-TFR-5.01-0520.}
\thanks{$^\dagger$Corresponding author}
\thanks{Data sharing does not apply to this article as no data sets were generated or analyzed during the current study.}
\begin{document}
\begin{abstract}
    Let $\Omega \subset \mathbb{R}^d$ with $d\geq 2$ be a bounded domain of class $\C^{1,\beta }$ for some $\beta \in (0,1)$. For $p\in (1, \infty )$ and $s\in (0,1)$, let $\Lambda ^s_{p}(\Omega )$ be the first eigenvalue of the mixed local-nonlocal operator $-\Delta _p+(-\Delta _p)^s$ in $\Omega $ with the homogeneous nonlocal Dirichlet boundary condition. We establish a strict Faber-Krahn type inequality for $\Lambda _{p}^s(\cdot )$ under polarization. As an application of this strict inequality, we obtain the strict monotonicity of  $\Lambda _{p}^s(\cdot )$ over annular domains and characterize the rigidity property of the balls in the classical Faber-Krahn inequality for $-\Delta _p+(-\Delta _p)^s$.
\end{abstract}
\maketitle
\section{Introduction}
For $p \in (1, \infty )$ and $s \in (0,1)$, we consider the following nonlinear eigenvalue problem on a bounded domain $\Omega \subset \mathbb{R}^d$ with homogeneous Dirichlet boundary condition:
\begin{equation}\label{Eigen_eqn}
    \begin{aligned}
        \mathcal{L}_p^s u&= \Lambda \abs{u}^{p-2} u  \mbox{ in } \Omega ,\\
        u&=0 \mbox{ in } \mathbb{R}^d\setminus \Omega,
    \end{aligned}
\end{equation}
where $\Lambda \in \mathbb{R}$ and $\mathcal{L}_p^s := -\Delta _p + (-\Delta _p)^s$ is a mixed local-nonlocal operator associated with the classical $p$-Laplace operator $\Delta _p u = {\rm div}(|\nabla u|^{p-2}\nabla u)$, and the fractional $p$-Laplace operator 
\begin{equation*}
    (-\Delta _p)^s u (x) = 2 \, \textrm{P.V.}\int _{\mathbb{R}^d} \frac{|u(x)-u(y)|^{p-2}(u(x)-u(y))}{|x-y|^{d+ps}} \,\mathrm{d}y
\end{equation*}
where P.V. denotes the principle value of the integral. We consider the following Sobolev spaces
\begin{align*}
     &W^{1,p}_0(\Omega) := \{u\in L^p(\mathbb{R}^d ) : \norm{\nabla u}_{L^p(\mathbb{R}^d )} < \infty \text{ with } u=0 \text{ in }\mathbb{R}^d \setminus \Omega \},\\ 
     & W_0^{s,p}(\Omega) := \{ u \in L^{p} (\mathbb{R}^d ) : [u]_{s,p} < \infty \text{ and } u = 0 \text{ in } \mathbb{R}^d \setminus \Omega\},
\end{align*}
where $\norm{\nabla \cdot}_{L^p(\mathbb{R}^d )}$ is the $L^p$-norm of the gradient, and $[\cdot ]_{s,p}$ is the Gagliardo seminorm in $\mathbb{R}^d$ that are given by 
\begin{align*}
    \norm{\nabla u}_{L^p(\mathbb{R}^d )}=\left(\int _{\mathbb{R}^d } |\nabla u|^p \dx \right)^{\frac{1}{p}} \; \mbox{ and } \; [u]_{s,p}= \left( \int _{\mathbb{R}^d}\int _{\mathbb{R}^d} \frac{|u(x)-u(y)|^p}{|x-y|^{d+sp}}\dy \dx \right)^{\frac{1}{p}}.
\end{align*}
In \cite[Lemma 2.1]{BSM2022}, the authors have shown that 
\begin{align*}
    \displaystyle [u]_{s,p} \le C(d,s,p, \Omega) \norm{\nabla u}_{L^p(\Omega )} \; \text{ for every } u \in W_0^{1,p}(\Omega).
\end{align*}
By this embedding we consider the solution space for~\eqref{Eigen_eqn} to be $W^{1,p}_0(\Omega)$. We say that $u \in W_0^{1,p}( \Omega)$ is a weak solution to~\eqref{Eigen_eqn} if it satisfies the following identity:
\begin{multline}\label{weak}
    \int_{\Omega } \abs{ \Gr u(x)}^{p-2} \Gr u(x) \cdot \Gr \phi(x) \, \dx + \int _{\mathbb{R}^d} \int _{\mathbb{R}^d} \frac{|u(x)-u(y)|^{p-2}(u(x)-u(y))}{\abs{x-y}^{d+ps}} (\phi(x)-\phi(y))\, \dy \dx  \\
    = \Lambda \int_{\Om} \abs{ u(x) }^{p-2} u(x) \phi(x) \, \dx \text{ for any } \phi \in  W_0^{1,p}( \Omega).
\end{multline}
In addition, if $u$ is non-zero, the real number $\Lambda$ is called an eigenvalue of~\eqref{Eigen_eqn}, and $u$ is a corresponding eigenfunction. From the classical and factional Poincar\'e inequalities~\cite[Lemma~2.4]{BrLiPa}, the least eigenvalue for~\eqref{Eigen_eqn} exists, and it is given by the following variational characterization 
\begin{equation}\label{var-char-Ineq-Mixed}
    \Lambda _{\mathcal{L}_p^s}(\Omega ) := \inf \left\{\norm{\nabla u}_{L^p(\Omega )}^p + [u]^p_{s,p}: u \in W_0^{1,p}(\Omega) \setminus \{0\} \text{ with } \|u\|_{L^p(\Omega )}=1\right\}.
\end{equation}
By the standard variational methods, there exists a minimizer $u\in W_0^{1,p}(\Omega) \setminus \{0\}$ for~\eqref{var-char-Ineq-Mixed}, and it is an eigenfunction of~\eqref{Eigen_eqn} corresponding to the eigenvalue $\Lambda _{\mathcal{L}^s_p}(\Omega )$. Without loss of generality, we assume that $u$ is non-negative, since $\abs{v} \in W^{1,p}_0( \Omega)$ and $\norm{ \Gr (\abs{v})}_{L^p(\Omega )}+[\abs{v}]_{s,p}  \leq \norm{ \Gr v}_{L^p(\Omega )}+[v]_{s,p}$ holds for every $v \in W^{1,p}_0( \Omega)$.  In Proposition~\ref{SMP}, we prove the strong maximum principle for $\mathcal{L}_p^s$, further implying that $u$ is strictly positive in $\Omega $.  

In this paper, we are concerned with the Faber-Krahn inequality for $\mathcal{L}^s_p$. The classical Faber-Krahn inequality for $\Delta _p$ states that: \emph{for any bounded open set $\Omega \subset \mathbb{R}^d$,}
\begin{equation}\label{Classical-FK}
    \lambda _1(\Omega ^*) \leq \lambda _1(\Omega ),
\end{equation}
\emph{with the equality only for $\Omega =\Omega ^*$ (up to a translation),} where $\lambda _1(\Omega )$ is the first Dirichlet eigenvalue of $-\Delta _p$ on $\Omega $, and $\Omega ^*$ is the ball centered at the origin in $\mathbb{R}^d$ with the same Lebesgue measure as that of $\Omega $. The \emph{rigidity property of the balls} (i.e., the balls are the unique sets realizing the equality in the Faber-Krahn inequality~\eqref{Classical-FK}) is discussed in~\cite{AFT98, Bhatta99, Anisa15, DanersKennedy, Kesavan88}. For $p\in (1,\infty )$ and $s\in (0,1)$, the Faber-Krahn inequality for the nonlocal operator $(-\Delta _p)^s$ is demonstrated in~\cite[Theorem~3.5]{BrLiPa}. Moreover, the Faber-Krahn inequality for the mixed local-nonlocal operator $\mathcal{L}^s_p$ is proved in~\cite[Theorem~1.1]{BSVV2023} when $p=2$, and in~\cite[Theorem~4.1]{BDVV23} when $p \neq 2$. To obtain the rigidity of this inequality, the authors of~\cite{BrLiPa, BSVV2023, BDVV23} used a strict rearrangement inequality for the seminorm $[\cdot ]_{s,p}$ involving the Schwarz symmetrization, as demonstrated in~\cite[Theorem~A.1]{Frank-Robert08} under the dimension restriction $d>ps$. Observe that $\Omega ^*$ is the Schwarz symmetrization of the domain $\Omega $, and the P\'olya-Szeg\"o inequality~\cite[Page 4818]{Beckner92} for the functions in $W^{1,p}_0(\Omega )$ under the Schwarz symmetrization gives the Faber-Krahn inequality. In~\cite{V2006, BrockSolynin2000}, the authors have shown that the Schwarz symmetrization of an $L^p(\mathbb{R}^d)$-function can be approximated by a sequence of polarizations. Polarization is a simple symmetrization in $\mathbb{R}^d$, also called a two-point rearrangement. Motivated by this approximation, one of the objectives of this paper is to generalize the Faber-Krahn inequality for $\mathcal{L}^s_p$ under polarization and also characterize the equality case. In addition, we show the rigidity property of the balls in the classical Faber-Krahn inequality for $\mathcal{L}^s_p$, using this characterization, without any restriction on the parameters $d, p,$ and $s$. Recently, the strict Faber-Krahn inequality under polarization is established in~\cite{Anoop-Ashok23} for $-\Delta _p$ where $p>\frac{2d+2}{d+2}$, and in~\cite{AK-NB2023Mono} for $(-\Delta _p)^s$ where $p\in (1,\infty )$ and $s\in (0,1)$.  

Now, we recall the notion of polarization for sets in $\mathbb{R}^d$, introduced by Wolontis~\cite{Wolontis}. An open affine-halfspace in $\mathbb{R}^d$ is called a~\emph{polarizer}. The set of all polarizers in $\mathbb{R}^d$ is denoted by $\mathcal{H}$. We observe that, for any polarizer $H\in \mathcal{H}$ there exist $h\in \mathbb{S}^{d-1}$ and $a\in \mathbb{R}$ such that $H = \left\{x\in \mathbb{R}^d : x\cdot h <a \right\}.$ For $H\in \mathcal{H}$, let $\sigma _H$ be the reflection in $\mathbb{R}^d$ with respect to $\partial H$, i.e., $\sigma _H(x) = x -2(x\cdot h -a)h$ for any $x\in \mathbb{R}^d$. The reflection of $A\subseteq \mathbb{R}^d$ with respect to $\partial H$ is $\sigma _H(A)=\{\sigma _H(x): x\in A\}$. It is straightforward to verify $\sigma _H(A\cup B)=\sigma _H(A)\cup \sigma _H(B),$ and $\sigma _H(A\cap B)=\sigma _H(A)\cap \sigma _H(B)$  for any $A, B\subseteq \mathbb{R}^d$. Now, we define the polarization of sets in $\mathbb{R}^d$, see~\cite[Definition~1.1]{Anoop-Ashok23}.
\begin{definition}
Let $H\in \mathcal{H}$ be a polarizer and $\Omega \subseteq \mathbb{R}^d$. The polarization $P_H(\Omega )$ of $\Omega $ with respect to $H$ is defined as
\begin{align*}
    P_H(\Omega ) &= \left[(\Omega \cup \sigma _H(\Omega ))\cap H\right] \cup [\Omega \cap \sigma _H(\Omega )].
\end{align*}
\end{definition}
It is easily observed that $P_H$ takes open sets to open sets, and $P_H$ is a rearrangement (i.e., it respects the set inclusion and preserves the measure) on $\mathbb{R}^d$. Now, we state our main results.
\begin{theorem}\label{Thm-unified} 
Let $p \in (1, \infty)$ and $s \in (0,1)$. Let $H\in \mathcal{H}$ be a polarizer, and $\Omega \subset \mathbb{R}^d$ be a bounded domain. Then 
        \begin{equation}\label{Polarized FK}
            \Lambda _{\mathcal{L}^s_p}(P_H(\Omega )) \leq \Lambda _{\mathcal{L}^s_p}(\Omega ).
        \end{equation}
     Further, if $\Omega $ is of class $\mathcal{C}^{1, \beta}$ for some $\beta \in (0,1)$ and $\Omega \neq P_H(\Omega ) \neq \sigma _H(\Omega )$ then 
        \begin{equation}\label{StrictFKtype}
            \Lambda _{\mathcal{L}^s_p}(P_H(\Omega )) < \Lambda _{\mathcal{L}^s_p}(\Omega ).
        \end{equation}
\end{theorem}

The effect of polarization on the $L^p$-norms of a function and its gradient and the Gagliardo seminorm (see Proposition~\ref{norm-properties}) prove the Faber-Krahn type inequality for $\Lambda _{\mathcal{L}^s_p}(\cdot )$. To establish the strict Faber-Krahn type inequality~\eqref{StrictFKtype}, we demonstrate a version of the strong comparison principle involving a Sobolev function and its polarization (see Proposition~\ref{SCP}). When $\Omega \neq P_H(\Omega )\neq \sigma _H(\Omega )$ and equality holds in the Faber-Krahn type inequality~\eqref{StrictFKtype}, we show a contradiction to this strong comparison principle in $\Omega \cap H$ with the help of two sets $A_H(\Omega )$ and $B_H(\Omega )$ with nonempty interiors (see Proposition~\ref{Propo-ball}-(iv)).   

As an application of Theorem~\ref{Thm-unified}, we get the strict monotonicity of the least eigenvalue $\Lambda _{\mathcal{L}^s_p}(\cdot )$ over annular domains. Precisely, we show that $\Lambda _{\mathcal{L}^s_p}\left(B_R(0)\setminus \overline{B}_r(s e_1)\right)$ strictly decreases for $0<s<R-r$ and $0<r<R<\infty $ (see Theorem~\ref{strictmonotonicity}).  In the subsequent theorem, using the variational characterization of $\Lambda _{\mathcal{L}^s_p}(\cdot )$ and the P\'olya-Szeg\"o inequalities under polarization (Proposition~\ref{norm-properties}), we derive the classical Faber-Krahn inequality for $\mathcal{L}^s_p$. Further, using Theorem~\ref{strictmonotonicity}, we prove the rigidity property of the balls in the classical Faber-Krahn inequality for $\mathcal{L}^s_p$.
\begin{theorem}\label{FK inequality}
Let $p \in (1, \infty)$ and $s \in (0,1)$. Let $\Omega \subset \mathbb{R}^d$ be a bounded domain and $\Omega ^*$ be the open ball centered at the origin in $\mathbb{R}^d$ such that $|\Omega ^*|=|\Omega |$. Then
\begin{equation}\label{classical FK}
    \Lambda _{\mathcal{L}^s_p}(\Omega ^*)\leq \Lambda _{\mathcal{L}^s_p}(\Omega ).
\end{equation}
Further, assume that $\Omega $ is of class $\mathcal{C}^{1, \beta}$ for some $\beta \in (0,1)$. Then, the equality occurs in~\eqref{classical FK} if and only if $\Omega =\Omega ^*$ (up to a translation).
\end{theorem}

The rest of the paper is organized as follows. Section~\ref{Section~2} discusses the regularity results and the strong maximum and strong comparison principles for the mixed operator $\mathcal{L}^s_p$. The strict Faber-Krahn inequality under polarization and the strict version of the classical Faber-Krahn inequality are proved in the final section.

\section{Regularity, Strong Maximum, and Strong Comparison Principles}\label{Section~2}

We begin the section by investigating the regularity of the eigenfunction for~\eqref{Eigen_eqn}. Subsequently, we prove a strong maximum principle for the mixed operator $\mathcal{L}_p^s$. Finally, we establish a variant of the strong comparison principle involving a function and its polarization. 

\subsection*{Regularity of the eigenfunctions} In this subsection, we obtain the $\mathcal{C}^{0,\alpha }$-regularity of the eigenfunctions of~\eqref{Eigen_eqn}.
We first prove that eigenfunctions lie in $L^{\infty}(\Omega)$. 
In~\cite[Theorem 4.1]{Biagi2024}, the authors obtained a similar result which says that if $u_0 \in W_0^{1,p}(\Omega)$ weakly satisfies $\mathcal{L}_p^s u =f(x,u)$ in $\Omega$, then $u_0 \in L^{\infty}(\Omega)$, where $f$ is a Carath\'eodory function satisfying certain growth assumptions. 

\begin{proposition}\label{regular}
    Let $p \in (1, \infty)$ and $s \in (0,1)$. Let $\Om \subset \R^d$ be a bounded open set, and $u \in W_0^{1,p}(\Omega)$ be a non-negative weak solution of~\eqref{Eigen_eqn}. Then $u \in L^{\infty}(\mathbb{R}^d )$. Further, if $\Om$ is of class $\mathcal{C}^{1,\beta}$ for some $\beta \in (0,1)$, then $u \in \C^{0,\al}(\mathbb{R}^d )$ for every $\al \in (0, 1)$.
\end{proposition}

\begin{proof}
To prove $u \in L^{\infty}(\mathbb{R}^d)$, we consider the case when $d>p$. For $ps < d \leq p$, our proof follows by employing analogous arguments with appropriate adjustments in the Sobolev embeddings of $W^{1,p}_0( \Omega)$. Set $M \geq 0$ and $\sigma > 1$. Consider $u_M:=\min\{u,M\}$ and $\phi:=u_M^{\sigma}$. Then $u_M, \phi \in L^{\infty}(\Omega ) \cap W_0^{1,p}(\Omega)$. We choose $\phi $ as a test function in the weak formulation~\eqref{weak}. Now, we calculate 
    \begin{align*}
        \abs{\Gr u}^{p-2} \Gr u \cdot \Gr ( u_M^{\sigma} ) = \sigma u_M^{\sigma -1 } \abs{\Gr u_M}^p = \sigma \left( \frac{p}{\sigma + p -1} \right)^p \left|  \Gr u_M^{\frac{\sigma + p -1}{p}} \right|^p. 
    \end{align*}
Observe that $u_M^{\frac{\sigma + p -1}{p}} \in W_0^{1, p }(\Omega)$, since $\sigma > 1$. Therefore, the Sobolev embedding $W_0^{1,p}(\Omega) \hookrightarrow L^{p^*}(\Omega )$ yields
\begin{equation}\label{lower_estimate 1}
    \begin{split}
        \int_{\Omega } \abs{ \Gr u(x)}^{p-2} \Gr u(x) \cdot \Gr \phi(x) \, \dx & =  \sigma \left( \frac{p}{\sigma + p -1} \right)^p \int_{\Omega } \left|  \Gr u_M(x) ^{\frac{\sigma + p -1}{p}} \right|^p \, \dx \\
        & \ge C(d,p) \sigma \left( \frac{p}{\sigma + p -1} \right)^p \left( \int_{ \Omega } u_M(x)^{\frac{p^*(\sigma + p -1)}{p}} \, \dx \right)^{ \frac{p}{p^*}},
    \end{split}
\end{equation}
where $C(d,p)$ is the embedding constant. Using the following inequality from~\cite[Lemma C.2]{BrLiPa}: 
    \begin{equation*}
        |a-b|^{p-2}(a-b)(a_M^\sigma -b_M^\sigma) \geq \frac{\sigma p^p}{(\sigma +p-1)^p}\left|a_M^{\frac{\sigma +p-1}{p}}-b_M^{\frac{\sigma +p-1}{p}}\right|^p,
    \end{equation*}
    (where $a, b\geq 0$, $a_M=\min\{a,M\}$, and $b_M=\min\{b,M\}$), together with the embedding $\wps \hookrightarrow  L^{p^*_s}(\mathbb{R}^d )$, we estimate the second integral on the left-hand side of~\eqref{weak} as follows:
  \begin{equation}\label{lower_estimate 2}
      \begin{split}
        \int _{\mathbb{R}^d} \int _{\mathbb{R}^d} &\frac{|u(x)-u(y)|^{p-2}(u(x)-u(y))}{\abs{x-y}^{d+ps}} (\phi(x)-\phi(y))\, \dy \dx \\
        & \geq \frac{\sigma p^p}{(\sigma + p-1)^p} \int_{\mathbb{R}^d } \int_{\mathbb{R}^d }  \frac{\left|u_M(x)^{\frac{\sigma+p-1}{p}}-u_M(y)^{\frac{\sigma + p - 1}{p}}\right|^{p}}{\abs{x-y}^{d+ps}}  \, \dx \\ 
        & \geq \frac{C(d,s,p)\sigma p^p}{(\sigma+p-1)^p} \left(\int_{\mathbb{R}^d } u_M(x)^{\frac{p_s^*(\sigma + p -1)}{p}} \,\dx\right)^{\frac{p}{p^*_s}}.
      \end{split}
  \end{equation}
Combining~\eqref{weak},~\eqref{lower_estimate 1}, and~\eqref{lower_estimate 2}, we obtain 
\begin{equation*}
    \begin{split}
        &  \frac{C(d,p)\sigma p^p}{(\sigma+p-1)^p} \left( \int_{ \Omega } u_M(x)^{\frac{p^*(\sigma + p -1)}{p}} \, \dx \right)^{ \frac{p}{p^*}} \\
        & + \frac{C(d,s,p)\sigma p^p}{(\sigma+p-1)^p} \left(\int_{\mathbb{R}^d } u_M(x)^{\frac{p_s^*(\sigma + p -1)}{p}}\,\dx\right)^{\frac{p}{p^*_s}} \le \Lambda \int_{\Om} u(x)^{\sigma + p-1} \, \dx.
    \end{split}
\end{equation*}
Since $M$ is arbitrary, the dominated convergence theorem yields
\begin{equation}\label{bound1}
    \begin{split}
        & \frac{C(d,p)\sigma p^p}{(\sigma+p-1)^p} \left( \int_{ \Omega } u(x)^{\frac{p^*(\sigma + p -1)}{p}} \, \dx \right)^{ \frac{p}{p^*}} \\
        & + \frac{C(d,s,p)\sigma p^p}{(\sigma+p-1)^p} \left(\int_{\mathbb{R}^d } u(x)^{\frac{p_s^*(\sigma + p -1)}{p}}\,\dx\right)^{\frac{p}{p^*_s}} \le \Lambda \int_{\Om} u(x)^{\sigma + p-1} \, \dx.
    \end{split}
\end{equation}
By observing that the first quantity of the above inequality is non-negative, we hence get 
\begin{align*}
    \frac{C(d,s,p)\sigma p^p}{(\sigma+p-1)^p} 
        \left(\int_{\mathbb{R}^d } u(x)^{\frac{p_s^*(\sigma + p -1)}{p}} \,\dx\right)^{\frac{p}{p^*_s}} \le \Lambda \int_{\Om} u(x)^{\sigma + p-1} \, \dx.
\end{align*}
The rest of the proof follows the identical arguments presented in~\cite[Theorem 3.3]{BrLiPa}. In the case $d = ps$, using the following Sobolev embedding
\begin{equation*}
\Theta_{s_1,p}(\Omega)
\left(\int_{\mathbb{R}^d }\left(u_M(x)^{\frac{\sigma+p-1}{p}}\right)^{2 p}\,\dx\right)^{\frac{1}{2}} \le \int _{\mathbb{R}^d} \int _{\mathbb{R}^d} \frac{\left|u_M(x)^{\frac{\sigma+p-1}{p}}-u_M(y)^{\frac{\sigma + p - 1}{p}}\right|^{p}}{\abs{x-y}^{d+ps}} \, \dy \dx,
\end{equation*}
where $
\Theta_{s,p}(\Omega):=\min\limits_{u\in W_0^{s,p}(\Omega)}\left\{[u]^p_{s,p}:\norm{u}_{L^{2p}(\Omega)}=1\right\},
$ 
we similarly get 
\begin{align}\label{form 1.1}
    \frac{ \Theta_{s,p}(\Omega) \sigma p^p}{(\sigma + p-1)^p} \left(\int_{\mathbb{R}^d } u(x)^{2 (\sigma + p -1)}\,\dx\right)^{\frac{1}{2}} \le \Lambda \int_{\mathbb{R}^d } u(x)^{\sigma + p-1} \, \dx,
\end{align}
and then rest the proof follows using~\cite[Theorem 3.3]{BrLiPa}. For $d<ps$, $u \in L^{\infty}(\mathbb{R}^d )$ follows from Morrey's inequality. Now, using the boundedness of $u$ and applying the regularity result~\cite[Theorem 1]{DM2022} to the equation~\eqref{Eigen_eqn} over $\mathcal{C}^{1,\beta}$-class domain $\Omega$, we conclude that $u \in \C^{0, \al}(\mathbb{R}^d )$ for every $\al \in (0,1)$.
\end{proof}

\subsection*{Strong maximum principle}
Now we discuss the strong maximum principle for $\mathcal{L}_p^s$. A strong maximum principle for the equation $\mathcal{L}_p^s u = V(x) \abs{u}^{p-2} u$, where $V \in L_{loc}^1(\Omega)$ with $V \le 0$ has been recently established in~\cite[Theorem 1.2]{Shang-Zhang23}. Also, we refer to~\cite[Theorem 3.1]{Biagi2024} for related strong maximum principle results. Our proof follows a similar procedure as in~\cite[Theorem 1.2]{Shang-Zhang23}. We require the following logarithmic energy estimate to demonstrate the strong maximum principle for~\eqref{Eigen_eqn}. 
\begin{lemma}[Logarithmic energy estimate]
    Let $p \in (1, \infty)$ and $s \in (0,1)$. Let $u \in W^{1,p}_0( \Omega)$ and $u \ge 0$ a.e. in $\Omega$. Assume that $u$ satisfies the following inequality: 
    \begin{align}\label{weak formulation}
        \int_{\Omega } \abs{ \Gr u}^{p-2} \Gr u \cdot \Gr \phi \dx & + \int _{\mathbb{R}^d} \int _{\mathbb{R}^d} \frac{|u(x)-u(y)|^{p-2}(u(x)-u(y))}{\abs{x-y}^{d+ps}} (\phi(x)-\phi(y))\, \dy \dx  \ge 0,
    \end{align}
    for every $\phi \in W^{1,p}_0(\Omega)$ with $\phi \ge 0$. Let $R>0$ and $x_0 \in \Omega$ be such that $B_R(x_0) \subset \Omega$. Then the following estimate holds for every $r> 0$ with $B_r(x_0) \subset B_{\frac{R}{2}}(x_0)$ and for every $\delta>0$: 
     \begin{multline}\label{log estimate}
         \int_{B_r(x_0)} \abs{ \Gr \log(u(x)+ \delta)}^p \, \dx +  \int _{B_r(x_0)} \int _{B_r(x_0)} \left|\log \left( \frac{u(x) + \delta}{u(y) + \delta} \right) \right|^p \, \frac{\dy \dx}{\abs{x-y}^{d+ps}} \\
             \leq C r^d \left( r^{-p} + r^{-ps} \right),
     \end{multline}
    where $C=C(d,p,s)>0$ is a universal constant.
\end{lemma}

\begin{proof}
Consider a cut-off function $\psi \in C_0^{\infty}(B_{\frac{3r}{2}})$ such that $\psi \equiv 1$ in $B_r(x_0)$ and $0 \le \psi \le 1$. Hence $\abs{ \Gr \psi} \le \frac{C}{r}$ in $B_{\frac{3r}{2}}(x_0)$ for some $C>0$. For $\delta>0$, we consider the function $\phi(x) = (u(x)+\delta)^{1-p} \psi(x)^p$ where $x \in \Omega$. We calculate 
\begin{align*}
    \Gr \phi = (1-p) \frac{\Gr u}{(u+ \delta)^p} \psi^p + p \frac{ \Gr \psi}{(u + \delta)^{p-1}} \psi^{p-1}.
\end{align*}
The above identity verifies that $ \phi \in W^{1,p}_0( \Omega)$. By choosing $\phi$ as a test function in~\eqref{weak formulation} we estimate the first integral of~\eqref{weak formulation} as follows:
\begin{align}\label{log-estimate 1}
        \int_{\Omega } \abs{ \Gr u}^{p-2} \Gr u \cdot \Gr \left( \frac{\psi^p}{(u+ \delta)^{p-1}} \right) 
        &= (1-p) \int_{ \Omega } \frac{\abs{\Gr u}^p}{ (u+ \delta)^p} \psi^p + p \int_{ \Omega} \frac{\abs{ \Gr u}^{p-2} \Gr u \cdot \Gr \psi }{ (u + \delta)^{p-1}} \psi^{p-1}.
\end{align}
Let $\varepsilon > 0$ be fixed. Applying Young's inequality, we get 
\begin{align*}
    \frac{\abs{\Gr u}^{p-1}}{(u + \delta)^{p-1}} \psi^{p-1} \abs{ \Gr \psi} \le \varepsilon \frac{\abs{\Gr u}^p}{ (u+ \delta)^p} \psi^p + C(\varepsilon) \abs{\Gr \psi}^p.
 \end{align*}
We choose $\varepsilon < \frac{p-1}{2}$, and also observe that $\abs{\Gr \log(u+\delta)}^p = \frac{\abs{ \Gr u}^p}{(u+ \delta)^p}$. Therefore,~\eqref{log-estimate 1} yields 
\begin{equation}\label{log-estimate 2}
    \begin{split}
        \int_{\Omega } \abs{ \Gr u}^{p-2} \Gr u \cdot \Gr \left( \frac{\psi^p}{(u+ \delta)^{p-1}} \right) \dx \leq \frac{1-p}{2} \int_{ \Omega } \abs{\Gr \log(u+\delta)}^p \dx + C(p) \int_{B_{\frac{3r}{2}}(x_0)} \abs{\Gr \psi}^p\dx.
    \end{split}
\end{equation}
Further, 
\begin{align*}
    \int_{B_{\frac{3r}{2}}(x_0)} \abs{\Gr \psi}^p \dx \le \frac{C^p}{r^p} |B_{\frac{3r}{2}}| = C(d) r^{d-p},
\end{align*}
for some $C(d)>0$.
Therefore, from~\eqref{log-estimate 2} we obtain 
\begin{align*}
    \int_{\Omega } \abs{ \Gr u}^{p-2} \Gr u \cdot \Gr \left( \frac{\psi^p}{(u+ \delta)^{p-1}} \right) \dx \leq -C(p) \int_{ \Omega } \abs{\Gr \log(u+\delta)}^p \dx + C(d) r^{d-p}.
\end{align*}
By~\cite[Lemma 1.3]{DKP2016}, we have the following estimate for the second integral of~\eqref{weak formulation}:
\begin{align}\label{log-estimate 3}
        &\int _{\mathbb{R}^d} \int _{\mathbb{R}^d} \frac{|u(x)-u(y)|^{p-2}(u(x)-u(y))}{\abs{x-y}^{d+ps}}  \left(  \frac{\psi(x)^p}{(u(x) + \delta)^{p-1}} -\frac{\psi(y)^p}{(u(y) + \delta)^{p-1}} \right) \, \dy \dx \no \\
        & \le C(d,p,s) \left( \; -\int _{B_{2r}(x_0)} \int _{B_{2r}(x_0)} \left|\log \left( \frac{u(x) + \delta}{u(y) + \delta} \right) \right|^p \psi(y)^p \, \frac{\dy \dx}{\abs{x-y}^{d+ps}} + r^{d-ps} \right),
\end{align}
where $C(d,p,s) >0$. For the inequality~\eqref{log-estimate 3}, we have used $u \ge 0$ a.e. in $\mathbb{R}^d $. Now by combining~\eqref{weak formulation},~\eqref{log-estimate 2}, and~\eqref{log-estimate 3}, and using the fact $\psi \equiv 1$ on $B_r(x_0)$, we obtain the desired estimate. 
\end{proof}

\begin{proposition}\label{SMP}
    Let $p \in (1, \infty)$, $s \in (0,1)$, and $\Omega \subset \mathbb{R}^d $ be a bounded domain.  Let $u \in W^{1,p}_0( \Omega)$ and $u \ge 0$ a.e. in $\Omega$. Assume that $u$ satisfies the following inequality: 
    \begin{align*}
        \int_{\Omega } \abs{ \Gr u(x)}^{p-2} \Gr u(x) \cdot \Gr \phi(x) \, \dx & + \int _{\mathbb{R}^d} \int _{\mathbb{R}^d} \frac{|u(x)-u(y)|^{p-2}(u(x)-u(y))}{\abs{x-y}^{d+ps}} (\phi(x)-\phi(y))\, \dy \dx  \ge 0
    \end{align*}
    for every $\phi \in W^{1,p}_0( \Omega)$ with $\phi \ge 0$. Then either $u \equiv 0$ or $u>0$ a.e. in $\Omega$. 
\end{proposition}

\begin{proof}
    We adapt the technique used in~\cite[Theorem A.1]{BF2014}. Let $K \subset \subset \Omega$ be any compact connected set. We first show that either $u \equiv 0$ or $u>0$ a.e. in $K$. Notice that $K \subset \{ x \in \Omega : \text{dist}(x, \partial \Omega) > 2r \}$ for some $r>0$. Since $K$ is compact, we choose $x_1,x_2,\dots, x_k$ in $\Omega$ such that $K \subset \cup_{i=1}^k B_{\frac{r}{2}}(x_i)$, and $\abs{B_{\frac{r}{2}}(x_i) \cap B_{\frac{r}{2}}(x_{i+1})} >0$ for each $i$, where $\abs{\cdot}$ denotes the Lebesgue measure. Suppose $u \equiv 0$ on a subset of $K$ with a positive measure. Then there exists $j \in \{1, \dots, k \}$ such that $\mathcal{A} = \{ x \in B_{\frac{r}{2}}(x_j) : u(x) = 0 \}$ has 
    a positive measure, i.e., $|\mathcal{A}| >0$. We define 
    \begin{align*}
        F_{\delta}(x) = \log \left( 1 + \frac{u(x)}{\delta} \right) \text{ for } x \in B_{\frac{r}{2}}(x_j).
    \end{align*}
    Clearly, $F_{\delta} \equiv 0$ on $\mathcal{A}$. Take $x \in B_{\frac{r}{2}}(x_j)$ and $y \in \mathcal{A}$ with $y \neq x$. Then 
    \begin{align*}
        \abs{F_{\delta}(x)}^p = \frac{ \abs {F_{\delta}(x) - F_{\delta}(y)}^p }{ \abs{x-y}^{d + ps} } \abs{x-y}^{d + ps}.
     \end{align*}
     Integrating the above identity over $\mathcal{A}$ we get
  \begin{align*}
      \abs{ \mathcal{A} } \abs{F_{\delta}(x)}^p \le \max_{x,y \in B_{\frac{r}{2}}(x_j)}{ \abs{x-y}^{d + ps} } \int_{ B_{\frac{r}{2}}(x_j)  } \left|\log \left( \frac{u(x) + \delta}{u(y) + \delta} \right) \right|^p \, \frac{\dy}{\abs{x-y}^{d+ps}}.
  \end{align*}
Further, the integration over $ B_{\frac{r}{2}}(x_j) $ yields
 \begin{align*}
     \int_{ B_{\frac{r}{2}}(x_j) } \abs{F_{\delta}(x)}^p \, \dx \le \frac{r^{d+ps}}{\abs{\mathcal{A}}} \int _{B_{\frac{r}{2}}(x_j)} \int _{B_{\frac{r}{2}}(x_j)} \left|\log \left( \frac{u(x) + \delta}{u(y) + \delta} \right) \right|^p \, \frac{\dy \dx}{\abs{x-y}^{d+ps}}.
 \end{align*}
     Now we use the above inequality and the logarithmic energy estimate~\eqref{log estimate} over $B_{\frac{r}{2}}(x_j)$ to get 
     \begin{equation*}
         \begin{split}
             &\int_{  B_{\frac{r}{2}}(x_j) } \abs{ \Gr \log(u(x)+ \delta)}^p \, \dx + \int_{ B_{\frac{r}{2}}(x_j) } \abs{F_{\delta}(x)}^p \, \dx \\
             &\le \int_{  B_{\frac{r}{2}}(x_j) } \abs{ \Gr \log(u(x)+ \delta)}^p \, \dx + \frac{r^{d+ps}}{\abs{\mathcal{A}}} \int _{B_{\frac{r}{2}}(x_j)} \int _{B_{\frac{r}{2}}(x_j)}
             \left|\log \left( \frac{u(x) + \delta}{u(y) + \delta} \right) \right|^p \, \frac{\dy \dx}{\abs{x-y}^{d+ps}} \\
             &\le C(d,p,s) r^d \left( r^{-p} + r^{-ps} \right) \max \left\{ 1, \frac{r^{d+ps}}{\abs{\mathcal{A}}} \right\}.
         \end{split}
     \end{equation*}
Thus, for every $ \delta > 0$, we have
    \begin{align*}
         \int_{ B_{\frac{r}{2}}(x_j) } \left|  \log \left( 1 + \frac{u(x)}{\delta} \right)  \right|^p \, \dx \le C(d,p,s) r^d \left( r^{-p} + r^{-ps} \right) \max \left\{ 1, \frac{r^{d+ps}}{\abs{\mathcal{A}}} \right\}.
    \end{align*}
    Letting $\delta \rightarrow 0$, from the above inequality we infer that $u \equiv 0$ a.e. in $B_{\frac{r}{2}}(x_j)$. Moreover, $u \equiv 0$ a.e. on a subset of positive measure in $B_{\frac{r}{2}}(x_{j+1})$ since $\abs{B_{\frac{r}{2}}(x_j) \cap B_{\frac{r}{2}}(x_{j+1})} >0$. Consequently, repeating the same arguments, we obtain $u \equiv 0$ a.e. in $B_{\frac{r}{2}}(x_{j+1})$, and then $u \equiv 0$ a.e. in $B_{\frac{r}{2}}(x_{i})$ for every $i = 1, \dots, k$. Thus $u \equiv 0$ a.e. in $K$. Hence, for every relatively compact set $K$ in $\Omega$, either $u \equiv 0$ or $u>0$ holds a.e. in $K$. Since $\Om$ is connected, there exists a sequence $(K_n)_{n \in \mathbb{N}}$ of compact sets such that $\abs{ \Omega \setminus K_n} \rightarrow 0$ as $n \rightarrow \infty$. Therefore, either $u \equiv 0$ or $u>0$ also holds a.e. in $\Omega$. This completes the proof.
  \end{proof}
  
\subsection*{Strong comparison principle}
 In this subsection, we discuss a strong comparison principle for the mixed local-nonlocal operator $\mathcal{L}^s_p$ involving a function and its polarization. For that, we recall the definition of the polarization of a function. 
 
\begin{definition}
    Let $H \in \mathcal{H}$ be a polarizer. For $u:\mathbb{R}^d\ra \mathbb{R}$, the polarization $P_H(u):\mathbb{R}^d \ra \mathbb{R}$ with respect to $H$ is defined as
\begin{equation*}
    P_H u(x) = \left\{
    \begin{aligned}
        &\max\{u(x), u(\sigma _H(x))\}, \quad \text{ for } x \in H,\\
        &\min\{u(x), u(\sigma _H(x))\}, \quad \text{ for } x \in \mathbb{R}^d\setminus H.
    \end{aligned}
    \right.
\end{equation*}
For $u:\Omega \ra \mathbb{R}$, let $\widetilde{u}$ be the zero extension of $u$ to $\mathbb{R}^d$. The polarization $P_H u:P_H(\Omega ) \ra \mathbb{R}$ is defined as the restriction of $P_H \widetilde{u}$ to $P_H(\Omega )$.
 \end{definition}
The polarization for functions on $\mathbb{R}^d$ is introduced by Ahlfors~\cite{Ahlfors73} (when $d=2$) and Baernstein-Taylor~\cite{Baernstein94} (when $d\geq 2)$. For further reading on polarizations and their applications, we refer the reader to~\cite{Anoop-Ash-Kesh, NirjanUjjalGhosh21, Brock04, Solynin12, Weth2010}. Next, in the spirit of~\cite[Definition~2.1]{BDVV22}, we define a weak solution of the following inequality involving $\mathcal{L}^s_p$: 
\begin{equation}\label{eqn-L_harmonic}
    \mathcal{L}^s_p u_1 - \mathcal{L}^s_p u_2 \geq 0,
\end{equation}
where $u_1, u_2 \in W^{1,p}(\mathbb{R}^d)$. We say a pair of functions $(u_1, u_2)$ solves~\eqref{eqn-L_harmonic} weakly in $\Omega $, if for any $\phi \in \mathcal{C}_c^\infty (\Omega )$ with $\phi \geq 0$,
\begin{multline*}
        \int_\Omega \left( |\nabla u_1|^{p-2}\nabla u_1  - |\nabla u_2|^{p-2}\nabla u_2 \right) \cdot \nabla \phi \dx 
        \\ + \int _{\mathbb{R}^d} \int _{\mathbb{R}^d} \frac{|u_1(x)-u_1(y)|^{p-2}(u_1(x)-u_1(y))}{|x-y|^{d+sp}} (\phi (x) -\phi (y)) \, \dy\dx 
        \\ -\int _{\mathbb{R}^d}\int _{\mathbb{R}^d} \frac{|u_2(x)-u_2(y)|^{p-2}(u_2(x)-u_2(y))}{|x-y|^{d+sp}} (\phi (x) -\phi (y)) \, \dy\dx \geq 0.
\end{multline*}
Now, we are ready to state a strong comparison type principle. For a set $A$, $|A|$ denotes the measure of $A$.
\begin{proposition}\label{SCP}
    Let $H\in \mathcal{H}$, and $\Omega \subset \mathbb{R}^d$ be an open set. Let $p\in (1, \infty )$, $s \in (0,1)$, and $u\in W_0^{1,p}(\Omega)$ be non-negative. Assume that $P_H u$ and $u$ satisfy the following equation weakly:
    \begin{equation}\label{SCP-eqn-Mixed}
        \mathcal{L}_p^s P_H u - \mathcal{L}_p^s u \geq 0 \mbox{ in } \Omega \cap H.
    \end{equation}
    Consider the following sets in $\Om \cap H$:
    \begin{align*}
      \mathcal{A} := \left\{ x \in  \Omega \cap H : P_H u(x) = u(x)  \right\}; \; \mathcal{B} := \left\{ x \in  \Omega \cap H : P_H u(x) > u(x)   \right\}.
    \end{align*}
    If $|\mathcal{B}|>0$, then $\mathcal{A}$ has an empty interior.
\end{proposition}
\begin{proof}
    Firstly, we denote $v=P_Hu$ in $P_H(\Omega )$, $w= v-u$ in $\Omega \cap H$, and $G(t)=|t|^{p-2}t$ with $G'(t)=(p-1)|t|^{p-2}\geq 0$ for $t\in \mathbb{R}$. On the contrary, we assume that $\mathcal{A}$ has a nonempty interior. Consider a test function $\phi \in \mathcal{C}_{c}^\infty (\mathcal{A})$ with $\phi > 0$ on $K := \text{supp}(\phi)$ where $|K|>0$. From~\eqref{SCP-eqn-Mixed} we have 
    \begin{multline}\label{scp-1}
        \int_\Omega \left(|\nabla v|^{p-2}\nabla v -|\nabla u|^{p-2}\nabla u\right)\cdot \nabla \phi (x) \dx  \\+\int_{\mathbb{R}^d } \int_{\mathbb{R}^d } \frac{G(v(x)-v(y))-G(u(x)-u(y))}{\abs{x-y}^{d+ps}} (\phi(x)-\phi(y))\, \dy \dx \geq 0.
    \end{multline}
    Since $\nabla \phi $ is supported in $\mathcal{A}$ and $\nabla v= \nabla u$ in $\mathcal{A}$, the first integral in~\eqref{scp-1} is zero, and hence
    \begin{equation}\label{scp-2}
        I:= \int_{\mathbb{R}^d } \int_{\mathbb{R}^d } \frac{G(v(x)-v(y))-G(u(x)-u(y))}{\abs{x-y}^{d+ps}} (\phi(x)-\phi(y))\, \dy \dx \geq 0.
    \end{equation}
    Using $\mathbb{R}^d=H\cup H^\mathsf{c}$ and the reflection $\overline{y}=\sigma _H(y)$, we write the inner integral of~\eqref{scp-2} as an integral over $H$ as below:
    \begin{align*}
        \ii{\mathbb{R}^d} &\frac{G(v(x)-v(y))-G(u(x)-u(y))}{\abs{x-y}^{d+ps}} (\phi(x)-\phi(y))\, \dy\\
        &= \left(\ii{H}+\ii{H^\mathsf{c}}\right) \frac{G(v(x)-v(y))-G(u(x)-u(y))}{\abs{x-y}^{d+ps}} (\phi(x)-\phi(y))\, \dy\\
        &=\ii{H} \left(\frac{1}{\abs{x-y}^{d+ps}}-\frac{1}{\abs{x-\overline{y}}^{d+ps}}\right) \big( G(v(x)-v(y))-G(u(x)-u(y)) \big) \phi(x)\, \dy\\
        &~~~-\ii{H} \frac{G(v(x)-v(y))-G(u(x)-u(y))}{\abs{x-y}^{d+ps}}\phi(y)\, \dy\\
        &~~~+\ii{H} \frac{G(v(x)-v(y))-G(u(x)-u(y))+G(v(x)-v(\overline{y}))-G(u(x)-u(\overline{y}))}{\abs{x-\overline{y}}^{d+ps}} \phi(x)\, \dy,
    \end{align*}
    where we use $\phi (\overline{y})=0$ for $y\in H$. We express the integral $I=I_1-I_2+I_3$ given by
    \begin{align*}
        I_1&= \int_{H} \int_{H} \left(\frac{1}{\abs{x-y}^{d+ps}}-\frac{1}{\abs{x-\overline{y}}^{d+ps}}\right) \big( G(v(x)-v(y))-G(u(x)-u(y)) \big) \phi(x)\, \dy \dx,\\
        I_2&= \int_{\mathbb{R}^d } \int_{H} \frac{G(v(x)-v(y))-G(u(x)-u(y))}{\abs{x-y}^{d+ps}}\phi(y)\, \dy \dx,\\
        I_3&= \int_{H} \int_{H} \frac{G(v(x)-v(y))-G(u(x)-u(y))+G(v(x)-v(\overline{y}))-G(u(x)-u(\overline{y}))}{\abs{x-\overline{y}}^{d+ps}} \phi(x)\, \dy \dx,
    \end{align*}
    where in $I_1$ and $I_3$ we use that the support of $\phi $ lies inside $H$. Using $\mathbb{R}^d=H\cup H^\mathsf{c}$,  $\overline{x}=\sigma _H(x)$, and $v(y)=u(y)$ for $y\in K={\rm supp} (\phi )$, we rewrite $I_2$ as 
    \begin{multline}\label{I_2 first}
        I_2=\int _{H} \int _{K} \frac{G(v(x)-u(y))-G(u(x)-u(y))}{\abs{x-y}^{d+sp}}\phi(y)\, \dy \dx \\
        + \int _{H} \int _{K} \frac{G(v(\overline{x})-u(y))-G(u(\overline{x})-u(y))}{\abs{\overline{x}-y}^{d+sp}}\phi(y)\, \dy \dx. 
    \end{multline}
    By the definition, either $v(x)=u(x)$ or $v(x)=u(\overline{x})$ for $x\in \mathbb{R}^d$. Consider the sets $K_1=\{x\in H : v(x)=u(\overline{x})\}$ and $K_2=\{x\in H : v(x)=u(x)\}$. Observe that
    \begin{align*}
        &v(x) = u(x) \text{ for some } x \in H \Longrightarrow u(\overline{x}) = v(\overline{x}), \text{ and } \\
        &v(x) = u(\overline{x}) \text{ for some } x \in H \Longrightarrow u(x) = v(\overline{x}).
    \end{align*}
Hence, both the terms in~\eqref{I_2 first} involving $G(\cdot)$ are zero on the set $K_2$, and from~\eqref{I_2 first} we get
    \begin{align*}
        I_2 &= \int _{K_1}\int _K \Bigl(G(v(x)-u(y))-G(u(x)-u(y))\Bigr)\left(\frac{1}{|x-y|^{d+sp}}-\frac{1}{|\overline{x}-y|^{d+sp}}\right)\phi(y) \, \dy\dx.
    \end{align*}
    Observe that $|x-y| < |\overline{x} - y|$ for any $x,y \in H$, and  $\phi \geq 0$ in $H$. Moreover, since $G(\cdot)$ is increasing and $v(x) \ge u(x)$ for $ x \in H$, we get $G(v(x)-u(y))-G(u(x)-u(y)) \geq 0$ for $x\in H$ and $y\in K$. Thus, we conclude that $I_2\geq 0.$ 
    Similarly, by replacing $v(x)=u(x)$ on $K$ we rewrite $I_3$ as 
    \begin{align*}
        I_3 &= \int_{K} \int _{H} \frac{G(u(x)-v(y))-G(u(x)-u(y))+G(u(x)-v(\overline{y}))-G(u(x)-u(\overline{y}))}{\abs{x-\overline{y}}^{d+sp}} \phi(x)\, \dy \dx.
    \end{align*}
    The term in the numerator involving $G(\cdot )$ is zero for $x\in K$ and $y \in H=K_1 \cup K_2$. Therefore, we get $I_3=0$. Thus,~\eqref{scp-2} implies that 
    \begin{align}\label{inequality 1}
        I_1= I+I_2 \geq 0.
    \end{align}
     Now we show that $I_1 < 0$, which will contradict~\eqref{inequality 1}. We have $v(x)-u(x)=0$ for $x \in K$. Hence for a.e. $x \in K$ and $y \in H$, we get  
    \begin{align*}
        (v(x)-v(y))-(u(x)-u(y)) = (v(x)-u(x)) - (v(y)-u(y)) \left\{\begin{array}{ll} 
             < 0, & \text {when } y \in \mathcal{B}; \\ 
             \le 0, & \text{when } y \in H \setminus \mathcal{B}. \\
             \end{array} \right.
    \end{align*}
    The monotonicity of $G(\cdot )$ gives:  for a.e. $x \in K$ and $y \in H$,
    \begin{align*}
        G(v(x)-v(y)) - G(u(x)-u(y)) \left\{\begin{array}{ll} 
             < 0, & \text {when } y \in \mathcal{B}; \\ 
             \le 0, & \text{when } y \in H \setminus \mathcal{B}. \\
             \end{array} \right.
    \end{align*}
    Further, $\phi > 0$ on $K$ and $|x-y| < |x-\overline{y}|$ for any $x,y \in H$. Therefore, we write $I_1$ as
    \begin{align*}
        &\left[ \;\int _{K} \int _{\mathcal{B}} + \int _{K} \int _{H \setminus \mathcal{B} }\right] \left(\frac{1}{\abs{x-y}^{d+sp}}-\frac{1}{\abs{x-\overline{y}}^{d+sp}}\right) \big[G(v(x)-v(y))-G(u(x)-u(y)) \big] \phi(x)\, \dy \dx \\
        &:= I_{1,1} + I_{1,2} < 0,
    \end{align*}
where the strict negativity follows from the facts $I_{1,1} < 0$ (as $|\mathcal{B}|>0$ and $\abs{K}>0$), and $I_{1,2} \le 0$. Therefore, we conclude that the set $\mathcal{A}$ has an empty interior. 
\end{proof}
\begin{remark}\label{rmk1}
    Our method to prove the above strong comparison principle can not be adapted to the purely local $p$-Laplace operator because the first integral in~\eqref{scp-1} associated with the $p$-Laplace operator over the set $\Omega \cap H$ is zero.   
\end{remark}
\section{Proof of main theorems}\label{Section 3} 
This section establishes a strict Faber-Krahn type inequality under polarization for $\Lambda _{\mathcal{L}^s_p}(\cdot )$. Then, using this strict Faber-Krahn type inequality, we obtain the classical Faber-Krahn inequality for $\Lambda _{\mathcal{L}^s_p}(\cdot )$. In the following proposition, we list some important properties of the polarization of sets, see~\cite[Proposition~2.2]{AK-NB2023Mono} or~\cite[Section~2]{Anoop-Ashok23}.

\begin{proposition}\label{Propo-ball}
    Let $H\in \mathcal{H}$ and $\Omega \subseteq \mathbb{R}^d.$ The following hold:
    \begin{enumerate}[label=\rm (\roman*)]
        \item $P_H(\sigma _H(\Omega ))=P_H(\Omega )$ and $P_{\sigma_H(H)}(\Omega ) = \sigma _H(P_H(\Omega ))$;
        \item $P_H(\Omega )\cap H$ is connected, if $\Omega $ is connected;
        \item $P_H(\Omega )=\Omega $ if and only if $\sigma _H(\Omega )\cap H\subseteq \Omega $;
        \item Further, assume that $\Omega $ is open. Then 
        \begin{enumerate}
            \item[\rm (a)] $A_H(\Omega ):= \sigma _H(\Omega ) \cap \Omega ^\mathsf{c} \cap H$ has a non-empty interior if and only if $P_H(\Omega )\neq \Omega $.
            \item[\rm (b)] $B_H(\Omega ):= \Omega \cap \sigma _H(\Omega ^\mathsf{c})\cap H$ has a non-empty interior if and only if $P_H(\Omega )\neq \sigma _H(\Omega )$.
        \end{enumerate}
    \end{enumerate}
\end{proposition}
Now, we list the invariance property of the $L^p$-norm of a function and its gradient under polarization~\cite{Anoop-Ashok23, Schaftingen05, Weth2010}, and the fractional P\'olya-Szeg\"o inequality under polarization, see~\cite[Page 4818]{Beckner92} or~\cite[Propsition~2.3]{AK-NB2023Mono}. 
\begin{proposition}\label{norm-properties}
    Let $H\in \mathcal{H}$, $\Omega \subseteq \mathbb{R}^d$ be open, and let $u:\Omega \ra [0,\infty )$ be measurable. If $u\in L^p(\Omega )$ for some $p\in [1,\infty )$, then $P_H(u)\in L^p(P_H(\Omega ))$ with $\left\|P_H(u)\right\|_{L^p(P_H(\Omega ))} = \left\|u\right\|_{L^p(\Omega )}$. Furthermore, if $u \in W^{1,p}_0(\Omega )$, then $P_H(u) \in W^{1,p}_0(P_H(\Omega ))$ with 
    \begin{equation*}
        \int _{P_H(\Omega )} |\nabla (P_Hu)|^p \dx = \int_\Omega |\nabla u|^p \dx,\mbox{ and } \left[P_H(u)\right]_{s,p} \leq \left[u\right]_{s,p}.
    \end{equation*}
\end{proposition}

\noi \textbf{Proof of Theorem~\ref{Thm-unified}.} Let $u\in W^{1,p}_0(\Omega )$ be a non-negative eigenfunction associated with $\Lambda _{\mathcal{L}^s_p}(\Omega )$ with $\left\|u\right\|_{L^p(\Omega )} =1$. By Proposition~\ref{norm-properties} and the variational characterization of $\Lambda _{\mathcal{L}^s_p}(\cdot )$, we get that $P_H(u) \in W^{1,p}_0(P_H(\Omega ))$ with $\left\|P_H(u)\right\|_{L^p(P_H(\Omega ))} = \left\|u\right\|_{L^p(\Omega )} =1$, and  
    \begin{equation}\label{ineq-Thm~3.1}
    \begin{split}
         \Lambda _{\mathcal{L}^s_p}(P_H(\Omega )) &\leq \norm{\nabla (P_Hu)}_{L^p(P_H(\Omega ))} + \left[P_H(u)\right]_{s,p} \leq \norm{\nabla u}_{L^p(\Omega)} + \left[u\right]_{s,p} = \Lambda _{\mathcal{L}^s_p}(\Omega ). 
    \end{split}
    \end{equation}
    Now, we prove the strict inequality~\eqref{StrictFKtype}. Assume that $\Omega $ is a $\mathcal{C}^{1,\beta }$-domain for some $\beta \in (0,1)$, and $\Omega \neq P_H(\Omega )\neq \sigma _H(\Omega )$. From Proposition~\ref{regular},  $u, P_H(u) \in \mathcal{C}(\mathbb{R}^d )$.  We consider the following sets
    \begin{align*}
    \mathcal{M}_u = \Bigl\{x\in P_H(\Omega )\cap H : P_H u(x) > u(x)\Bigr\} \text{ and } \mathcal{N}_u = \Bigl\{x\in P_H(\Omega )\cap H : P_H u(x) = u(x)\Bigr\}.
    \end{align*}
    Observe that $ \mathcal{M}_u$ is open and $\mathcal{N}_u$ is relatively closed in $P_H(\Omega )\cap H$, and further $P_H(\Omega )\cap H = \mathcal{M}_u \sqcup \mathcal{N}_u$. First, we find a set $B$ in the open set $\Omega \cap H$ such that $B\cap \mathcal{M}_u \neq \emptyset .$ 
    Using the assumption $\Omega \neq P_H(\Omega )\neq \sigma _H(\Omega )$, notice that both $A_H(\Omega )$ and $B_H(\Omega )$ have non-empty interiors (see Proposition~\ref{Propo-ball}-(iv)). Further, by the strong maximum principle (Proposition~\ref{SMP}), $u>0$ in $\Omega $. By the definition, we have $P_H u\geq u$ in $P_H(\Omega )\cap H$, and 
    \begin{equation}\label{eq-3.1}
        \begin{aligned}
        \text{in } A_H(\Omega ): &~~~ u=0, ~u\circ \sigma _H>0, \text{ and hence } P_H u=u\circ \sigma _H>u;\\
        \text{in } B_H(\Omega ): &~~~ u>0, ~u\circ \sigma _H=0, \text{ and hence } P_H u=u.
    \end{aligned}
    \end{equation}
    By~\eqref{eq-3.1}$, \mathcal{M}_u\supseteq A_H(\Omega )$, and $\mathcal{N}_u\supseteq B_H(\Omega )$.  Therefore, using $P_H(\Omega )\cap H =(\Omega \cap H)\sqcup A_H(\Omega ) = \mathcal{M}_u \sqcup \mathcal{N}_u$, $\mathcal{N}_u$ is relatively closed in  $P_H(\Omega )\cap H$, and $\Omega \cap H$ is an open set, we see that $\mathcal{N}_u \subsetneq \Omega \cap H$, because if $\mathcal{N}_u=\Omega \cap H$ then the connected set $P_H(\Omega )\cap H$ is a union of two disjoint open sets which is not possible.
    Hence, the set $B:=\mathcal{M}_u \cap \Omega \cap H$ is a nonempty open set. Therefore, the sets $\mathcal{N}_u, B \subset \Omega \cap H$ have the following properties:
    \begin{equation}\label{eq-ball}
        P_H u > u \text{ in } B 
        \text{ and }
        P_H u \equiv u \text{ in } \mathcal{N}_u.
    \end{equation}
    On the contrary to~\eqref{StrictFKtype}, assume that $\Lambda _{\mathcal{L}^s_p}(P_H(\Omega )) = \Lambda _{\mathcal{L}^s_p}(\Omega ) = \lambda .$ Now, the equality holds in~\eqref{ineq-Thm~3.1}, and hence $P_H(u)$ becomes a non-negative minimizer of the following problem:
    \begin{align*}
    \Lambda _{\mathcal{L}^s_p}(P_H(\Omega ) = \min_{v \in W_0^{1,p}(P_H(\Omega )} \left\{ \norm{\nabla v}_{L^p(P_H(\Omega ))}^p + \left[v\right]_{s,p}^p : \left\|v\right\|_{L^p(P_H(\Omega ))} =1 \right\}.  
    \end{align*}
    As a consequence, the following equation holds weakly:
    \begin{equation}\label{eqn_1}
    \mathcal{L}^s_p P_H(u) = \Lambda _{\mathcal{L}^s_p} (P_H(\Omega)) |P_H(u)|^{p-2}P_H(u) \mbox{ in } P_H(\Omega), \; P_H(u)=0 \mbox{ in } \mathbb{R}^d\setminus P_H(\Omega).
    \end{equation}
    Since $\Om \cap H \subset P_H(\Omega  \cap H$, both $u$ and $P_H u$ are weak solutions of the following equation:
    \begin{align*}
        \mathcal{L}_p^s w -\lambda |w|^{p-2} w =0 \text{ in } \Omega \cap H.
    \end{align*}
    Using $P_H u \geq u$ in $\Omega \cap H$, we see that the following holds weakly:
    \begin{align*}
        \mathcal{L}_p^s P_H(u) -\mathcal{L}_p^s u = \lambda (|P_H(u)|^{p-2}P_H(u) - |u|^{p-2} u) \geq 0 \text{ in } \Omega \cap H.
    \end{align*}
    Since  $B_H(\Omega ) \subseteq \mathcal{N}_u \subsetneq \Omega \cap H$, we see that $\mathcal{N}_u \subseteq \mathcal{A}$ with $\mathcal{N}_u$ having a nonempty interior. Moreover, $|\mathcal{B}| > 0$, since $B \subseteq \mathcal{B}$ and $B$ is a nonempty open set. Now, we apply the strong comparison type principle (Proposition~\ref{SCP}) to get a contradiction to~\eqref{eq-ball}. Therefore, the strict inequality~\eqref{StrictFKtype} holds.
\qed

In the spirit of~\cite[Theorem~3.3]{AK-NB2023Mono} and~\cite[Theorem~1.5]{Anoop-Ashok23}, we can apply the strict Faber-Krahn inequality (Theorem~\ref{Thm-unified}) to the family of annular domains $B_R(0)\setminus \overline{B}_r(t e_1)$ for $0<r<R$ and $0\leq t<R-r.$ For $0\leq t_1 <t_2<R-r$, let $a=\frac{t_1+t_2}{2}$ and $H_a:=\left\{ x=(x_1,x')\in \mathbb{R}^d : x_1 <a\right\}$. Then, as in~\cite[Proposition~3.2-(iii)]{AK-NB2023Mono}, we get $P_{H_a}(\Omega _{t_1})=\Omega _{t_2}$. Now, applying Theorem~\ref{Thm-unified}, we get the following result.
\begin{theorem}[Strict monotonicity of $\Lambda _{\mathcal{L}^s_p}(\cdot)$ over annular domains]\label{strictmonotonicity}
    Let $p \in (1, \infty),$ $s \in (0,1)$, and $0 < r <R$. Then $\Lambda _{\mathcal{L}^s_p}\left(B_R(0)\setminus \overline{B}_r(t e_1)\right)$ is strictly decreasing for $0\leq t < R-r$. In particular, 
    \begin{equation*}
        \Lambda _{\mathcal{L}^s_p}\left(B_R(0)\setminus \overline{B}_r(0)\right)=\max \limits _{0\leq t < R-r}\Lambda _{\mathcal{L}^s_p}\left(B_R(0)\setminus \overline{B}_r(t e_1)\right).
    \end{equation*}
\end{theorem}

Now, we consider the generalized mixed local-nonlocal operator 
\begin{align*}
    \mathcal{G}_p^s := - a \Delta _p + b (-\Delta _p)^s, \text{ with } a, b \ge 0,
\end{align*}
and consider the following eigenvalue problem as in~\eqref{Eigen_eqn}:
\begin{equation}\label{Eigen_eqn_general}
    \begin{aligned}
        \mathcal{G}_p^s u&= \Lambda \abs{u}^{p-2} u  \mbox{ in } \Omega ,\\
        u&=0 \mbox{ in } \mathbb{R}^d\setminus \Omega.
    \end{aligned}
\end{equation}
The solution space for the equation~\eqref{Eigen_eqn_general} is $W^{s, p}_0( \Omega)$ when $a =0$, and $W^{1, p}_0( \Omega)$ when $a>0$. We denote by $\mathbb{X}(\Omega)$ as the solution space for~\eqref{Eigen_eqn_general}. Arguing essentially as in the proof of Theorem~\ref{Thm-unified}, an analogous strict Faber-Krahn inequality holds for the first eigenvalue of~\eqref{Eigen_eqn_general}, as stated below.

\begin{theorem}\label{Thm-unified-general} 
Let $p \in (1, \infty)$, $s \in (0,1)$, $a \ge 0$, and $b >0$. Let $H$ be a polarizer, and $\Omega \subset \mathbb{R}^d$ be a bounded domain. Then, the least eigenvalue of~\eqref{Eigen_eqn_general} given by
\begin{equation*}
    \Lambda _{\mathcal{G}_p^s}(\Omega ) := \inf \left\{ a \norm{\nabla u}_{L^p(\Omega )}^p + b [u]^p_{s,p}: u \in \mathbb{X}(\Omega) \setminus \{0\} \text{ with } \|u\|_{L^p(\Omega )}=1\right\}
\end{equation*}
satisfies  $$\Lambda _{\mathcal{G}^s_p}(P_H(\Omega )) \leq \Lambda _{\mathcal{G}^s_p}(\Omega ).$$ Further, if $\Omega $ is of class $\mathcal{C}^{1,\beta }$ for some $\beta \in (0,1)$ and $\Omega \neq P_H(\Omega ) \neq \sigma _H(\Omega )$, then $$\Lambda _{\mathcal{G}^s_p}(P_H(\Omega )) < \Lambda _{\mathcal{G}^s_p}(\Omega ).$$
\end{theorem}

\begin{remark}
The above theorem does not consider the case $b =0$ due to the lack of a strong comparison principle involving $u$ and $P_H(u)$ (see Remark~\ref{rmk1}). Nevertheless, in the case $b =0$ and $p>\frac{2d+2}{d+2}$, using the classical strong comparison principle for the $p$-Laplace operator due to Sciunzi~\cite[Theorem~1.4]{Sciunzi2014}, a strict Faber-Krahn type inequality for $\Lambda _{\mathcal{G}_p^s}$ is proved in~\cite[Theorem~1.3]{Anoop-Ashok23}. 
\end{remark}

We state the following result, approximating the Schwarz symmetrization by a sequence of polarizers.
\begin{proposition}{\rm \cite[Theorem 4.4]{V2006}}\label{Sequence_pol}
    Let $*$ be the Schwarz symmetrization on $\mathbb{R}^d$. Then, there exists a sequence of polarizers $(H_m)_{m\in \mathbb{N}}$ in $\mathcal{H}$ with $0\in \partial H_m$ such that, for any $1 \leq p <\infty ,$ if $u \in L^p(\mathbb{R}^d)$ is non-negative, then the sequence
    \begin{align}
        u_m := P_{H_1\cdots H_m}(u)
    \end{align}
    converges to $u^*$ in $L^p(\mathbb{R}^d )$, i.e., 
    \begin{equation*}
        \lim_{m \ra \infty }\norm{u_m - u^*}_p =0.
    \end{equation*}
\end{proposition}

\noi \textbf{Proof of Theorem~\ref{FK inequality}:}
    Let $u\in W^{1,p}_0(\Omega )$ be the non-negative eigenfunction of~\eqref{Eigen_eqn} corresponding to $\Lambda _{\mathcal{L}^s_p}(\Omega )$ with $\norm{u}_{L^p(\Omega )}=1$. First, we prove the following P\'olya-Szeg\"o inequality for $u$:
    \begin{align}\label{P-Sz_General}
        \int _{\Omega ^*} |\nabla u^*|^p\dx + [u^*]_{s,p} \leq \int _\Omega |\nabla u|^p\dx + [u]_{s,p},
    \end{align}
    where $u^*$ is the Schwarz symmetrization of the function $u$ in $\mathbb{R}^d$. Let $(H_m)_{m\in \mathbb{N}} \subset \mathcal{H}_{*}:=\{H\in \mathcal{H} : 0\in \partial H\}$ be the sequence of polarizers, and $(u_m)_{m\in \mathbb{N}}$ be the sequence of functions as in Proposition~\ref{Sequence_pol}. Since $u\in L^p(\mathbb{R}^d)$ is non-negative, by 
    Proposition~\ref{Sequence_pol} we know that 
    \begin{itemize}
        \item $u_m \longrightarrow u^*$ in  $L^p(\mathbb{R}^d),$
        \item $u_m \longrightarrow u^*$ pointwise a.e. in $\mathbb{R}^d$  (up to a subsequence), and
        \item $\norm{ u_m }_{L^p(\mathbb{R}^d)} \longrightarrow \norm{ u^* }_{L^p(\mathbb{R}^d)}.$ 
    \end{itemize}
    Let $B_R(0)\subset \mathbb{R}^d $ be an open ball large enough such that $\Omega \subseteq B_R(0)$. Observe that, for every $H\in \mathcal{H}$ the support of $P_H(u)$ is a subset of  $P_H({\rm supp}(u))$ (see \cite[Proposition~2.14]{Anoop-Ashok23}). Now, for every $m \in \N$, we get $u_m =0$ in $\mathbb{R}^d \setminus \Omega _m$, where $\Omega _m := P_{H_1\cdots H_m}(\Omega ) \subseteq B_R(0)$, and hence $u_m \in W^{1,p}_0(B_R(0))$. Further, using Proposition~\ref{norm-properties}, we have
    \begin{align*}
        \int_{ \mathbb{R}^d } \abs{u_m}^p\dx = \int_{ \mathbb{R}^d } \abs{u}^p \dx=1, \text{ and }  \int _{\mathbb{R}^d}|\nabla u_m|^p\dx + [u_m]_{s,p} \leq \int _\Omega |\nabla u|^p\dx + [u]_{s,p}.
    \end{align*}
    Therefore, the sequence $(u_m)_{m\in \mathbb{N}}$ is bounded in $W_0^{1,p}(B_R(0))$. By the reflexivity of $W_0^{1,p}(B_R(0))$, we get $u_m \rightharpoonup \widetilde{u}$ in $W_0^{1,p}(B_R(0))$, and by the compactness of the embedding $W_0^{1,p}(B_R(0)) \hookrightarrow L^p( B_R(0))$, we have $u_m \rightarrow \widetilde{u}$ in $L^p(B_R(0))$. Hence, up to a subsequence, $u_m \ra \widetilde{u}$ pointwise a.e. in $B_R(0)$. Now, by the uniqueness of the limit, $\widetilde{u} = u^*$ a.e. in $B_R(0)$. Moreover, since $(u_m)_{m\in \mathbb{N}}$ is bounded in $W_0^{ s, p}(B_R(0))$, $u_m \rightharpoonup u^*$ in $W_0^{ s, p}(B_R(0))$ as well. Applying the weak lower semicontinuity of the norms in $W_0^{ 1, p}(B_R(0))$ and  $W_0^{ s, p}(B_R(0))$,  we obtain
    \begin{align*}
        &\int _{\mathbb{R}^d}|\nabla u^*|^p\dx + [u^*]_{s,p} \leq \liminf_{m \ra \infty} \left( \int _{\mathbb{R}^d}|\nabla u_m|^p\dx + [u_m]_{s,p} \right) \leq \int _\Omega |\nabla u|^p\dx + [u]_{s,p}.
    \end{align*}
    Since ${\rm supp}\,(u^*) \subseteq \Omega ^*$, we get the required inequality~\eqref{P-Sz_General}. Indeed $u^*\in W^{1,p}_0(\Omega ^*).$ Now,  
    the variational characterization~\eqref{var-char-Ineq-Mixed} and the inequality~\eqref{P-Sz_General} imply the Faber-Krahn inequality
    \begin{equation*}
        \Lambda _{\mathcal{L}^s_p}(\Omega ^*) \leq \Lambda _{\mathcal{L}^s_p}(\Omega ).
    \end{equation*}
    Next, assume that $\Omega $ is of class $\mathcal{C}^{1,\beta }$ for some $\beta \in (0,1)$. Suppose the equality $\Lambda _{\mathcal{L}^s_p}(\Omega ^*) = \Lambda _{\mathcal{L}^s_p}(\Omega )$ holds. To prove $\Omega$ is a ball, we first show that $\Omega $ is radial with respect to a point $x_0 \in \mathbb{R}^d$. Here $x_0\in \Omega$ is chosen such that $d(x_0,\partial \Omega )= \sup \{d(x, \partial \Omega ) : x\in \overline{\Omega }\}.$ Let $H\in \mathcal{H}_{x_0}:=\{H\in \mathcal{H} : x_0 \in \partial H\}$, and notice that $|P_H(\Omega )|=|\Omega |=|\Omega ^*|$. Using the Faber-Krahn inequalities~\eqref{ineq-Thm~3.1} and~\eqref{Polarized FK}, we get
    \begin{align*}
        \Lambda _{\mathcal{L}^s_p} (P_H(\Omega )) \leq \Lambda _{\mathcal{L}^s_p}(\Omega ) \text{ and }
        \Lambda _{\mathcal{L}^s_p}(\Omega ^*) \leq \Lambda _{\mathcal{L}^s_p}(P_H(\Omega )).
    \end{align*}
    Therefore, we have $\Lambda _{\mathcal{L}^s_p} (P_H(\Omega )) = \Lambda _{\mathcal{L}^s_p}(\Omega )$, and hence by Theorem~\ref{Thm-unified} we get
    \begin{equation}\label{eqn-3.8}
        P_H(\Omega )=\Omega \text{ or } P_H(\Omega )=\sigma _H(\Omega ) \text{ for every } H\in \mathcal{H}_{x_0}.
    \end{equation} 
    Recall that, \emph{`$\Omega $ is radial with respect to $x_0 \in \mathbb{R}^d$ if and only if $\sigma _H(\Omega )=\Omega $ for any $H\in \mathcal{H}_{x_0}$.'} On the contrary, assume that $\Omega $ is not radial with respect to $x_0$. Then, $\overline{\Omega }^\mathsf{c}\cap \sigma _H(\Omega )$ is a non-empty open set for some $H\in \mathcal{H}_{x_0}$. Therefore $\overline{\Omega }^\mathsf{c}\cap \sigma _H(\Omega )\cap \partial H \neq \emptyset $ as $x_0 \in \Omega $. Further, the sets $\overline{\Omega }^\mathsf{c}\cap \sigma _H(\Omega )\cap H \subseteq A_H(\Omega )$ and $\overline{\Omega }^\mathsf{c} \cap \sigma _H(\Omega )\cap \sigma_H({H})\subseteq \sigma _H(B_H(\Omega ))$ are also non-empty open sets. Hence, both $A_H(\Omega )$ and $B_H(\Omega )$ have non-empty interiors. Now, by Proposition~\ref{Propo-ball}-(iv), we get
    \begin{equation*}
        P_H(\Omega )\neq \Omega \text{ and } P_H(\Omega )\neq \sigma _H(\Omega ).
    \end{equation*}
    This contradicts~\eqref{eqn-3.8}, and therefore $\Omega $ is radial with respect to $x_0$. Since $\Omega $ is connected, we conclude that either $\Omega $ is a ball $\Omega ^*$ or a concentric annular domain $B_R(0)\setminus \overline{B}_r(0)$ for some $0<r<R$, up to a translation. Further, using the facts
    \begin{align*}
        |\Omega ^*|=|B_R(0)\setminus \overline{B}_r(0)|=|B_R(0)\setminus \overline{B}_r(t e_1)| \text{ for } 0\leq t <R-r,
    \end{align*}
    the Faber-Krahn inequality~\eqref{classical FK} implies that
    \begin{align*}
        \Lambda _{\mathcal{L}^s_p}(\Omega ^*) \leq \Lambda _{\mathcal{L}^s_p}\left(B_R(0)\setminus \overline{B}_r(t e_1)\right) \mbox{ for any } 0< t <R-r.
    \end{align*}
    Now, by the strict monotonicity of $\Lambda _{\mathcal{L}^s_p}(\cdot)$ over annular domains (Theorem~\ref{strictmonotonicity}), we obtain
    \begin{align*}
        \Lambda _{\mathcal{L}^s_p}(\Omega ^*) \leq \Lambda _{\mathcal{L}^s_p}\left(B_R(0)\setminus \overline{B}_r(t e_1)\right) < \Lambda _{\mathcal{L}^s_p}\left(B_R(0)\setminus \overline{B}_r(0)\right) \mbox{ for any } 0< t <R-r. 
    \end{align*}
    Therefore, $\Omega$ must be the ball $\Omega ^*$ (up to a translation).
\qed

For $a \geq 0$ and $b >0$, using Theorem~\ref{Thm-unified-general} and adapting the proof of Theorem~\ref{FK inequality}, we have the following classical Faber-Krahn inequality for $\mathcal{G}_p^s$ without any dimension restriction $d>sp$.
\begin{theorem}
Let $p \in (1, \infty)$, $s \in (0,1)$, $a \ge 0$, and $b >0$. For a bounded domain $\Omega \subset \mathbb{R}^d$ of class $\mathcal{C}^{1,\beta }$ with $\beta \in (0,1)$, let $\Omega ^*$ be the open ball centered at the origin in $\mathbb{R}^d$ such that $|\Omega ^*|=|\Omega |$. Then $\Lambda _{\mathcal{G}_p^s}(\Omega ^*)\leq \Lambda _{\mathcal{G}_p^s}(\Omega )$ with the equality occurs if and only if $\Omega =\Omega ^*$ (up to a translation).
\end{theorem}

 \noi \textbf{Acknowledgments.}
We thank the anonymous reviewer for their insightful comments. This work is partly funded by the Department of Atomic Energy, Government of India, under project no. 12-R$\&$D-TFR-5.01-0520.

\end{document}